\documentclass[a4paper,12pt]{amsart}
\usepackage{amssymb,enumerate,latexsym,verbatim}
\usepackage[T1]{fontenc}

\title{Epimorphisms between linear orders}
\author{Riccardo Camerlo}
\address{Dipartimento di scienze matematiche \guillemotleft{Joseph-Louis
Lagrange}\guillemotright, Politecnico di Torino, Corso Duca degli Abruzzi 24,
10129 Torino --- Italy}
\email{riccardo.camerlo@polito.it}
\author{Rapha\"el Carroy}
\address{Dipartimento di matematica \guillemotleft{Giuseppe Peano}\guillemotright,
Universit\`a di Torino, Via Carlo Alberto 10, 10121 Torino --- Italy}
\email{raphael.carroy@unito.it}
\author{Alberto Marcone}
\address{Dipartimento di matematica e informatica, Universit\`a di Udine,
Via delle Scienze 208, 33100 Udine --- Italy}
\email{alberto.marcone@uniud.it}

\thanks{The second and third authors acknowledge the support of PRIN 2009 Grant
\lq\lq Modelli e Insiemi\rq\rq\  number 2009WY32E8 004, and the second author was as well
partially supported by SFB Grant number 878.}

\theoremstyle{plain}
\newtheorem{theorem}{Theorem}
\newtheorem{lemma}[theorem]{Lemma}
\newtheorem{fact}[theorem]{Fact}
\newtheorem{facts}[theorem]{Facts}
\newtheorem{cor}[theorem]{Corollary}

\newtheorem{prop}[theorem]{Proposition}

\theoremstyle{definition}
\newtheorem{definition}[theorem]{Definition}
\newtheorem{definitions}[theorem]{Definitions}

\newtheorem*{nota}{Notation}

\newcommand{\sse}{\Leftrightarrow}
\newcommand{\dom}{\operatorname{dom}}

\newcommand{\rest}[1]{ |_{#1}}

\newcommand{\rao} {\rightarrow}

\newcommand{\lgrao} {\longrightarrow}

\newcommand{\Cal}[1]{\ensuremath{\mathcal{#1}}}

\newcommand{\st} {\text{such that }}

\newcommand{\sub}{\subseteq}

\DeclareMathOperator{\im}{im}

\newcommand{\scat}{Scat}
\newcommand{\lin}{Lin}
\newcommand{\Lin}{LIN}
\newcommand{\strong}{\leq_{s}}
\newcommand{\nstrong}{\nleq_{s}}
\newcommand{\inj}{\leq_{i}}
\newcommand{\injcont}{\leq_{c}}
\newcommand{\colorcplt}{\leq_{col}}
\newcommand{\color}{\leq_{cQ}}

\newcommand{\Pow}{\mathcal{P}}

\begin{document}
\pagestyle{plain}

\begin{abstract}
We study the relation on linear orders induced by order preserving
surjections. In particular we show that its restriction to countable orders
is a bqo.
\end{abstract}
\maketitle

\tableofcontents

\section{Some generalities and the questions}

Fra\"iss\'e (\cite{fraisse}) conjectured that the class of countable linear
orders was a well-quasi-order (wqo for short) under order preserving
injections (also called embeddings). Laver (\cite{laverfraisse}) proved that
this class is in fact a better-quasi-order (bqo for short), which is much
stronger.

We are interested in the somehow dual quasi-order (qo for short) induced by
order preserving surjections, or \emph{epimorphisms}, between linear orders,
in particular the countable ones. What are the combinatorial properties of
this qo?

In Section \ref{background} we present the basic definitions and facts about
linear orders, bqos and wqos, and epimorphisms. The main result of Section
\ref{countableorders} states that order preserving surjections induce a bqo
on countable linear orders; this is a consequence of a theorem of van
Engelen, Miller and Steel (\cite{vEMSbqo}) stating that the class of
countable linear orders with continuous order preserving injections preserves
bqos. We also look at the stronger notion of preserving bqos and show how to adapt it to epimorphisms in order to keep its validity in this setting. In
Section \ref{descriptionofspecialclasses} we apply the tools that we have
developed earlier to describe explicitly the relation of epimorphism on some
restricted classes of linear orders, such as ordinals.

\section{Background} \label{background}

\subsection{On linear orders}

\begin{definitions}~
\begin{itemize}
\item A \emph{linear order} is a reflexive, transitive, antisymmetric and
    total relation on a non-empty set $K$; we usually denote it with
    $\leq_K$ (or $\leq$ when no confusion arises) and we say (abusively)
    that $K$ is a linear order. We also write $<_K$ or $<$ for the strict
    part of the order.
\item Given a linear order $K$, a \emph{suborder} of $K$ is a subset of $K$
    along with the induced order on it.
\item A linear order $K$ is \emph{dense} when, given any elements $x,y$ in
    $K$, if $x< y$ holds then there is a $z\in K$ so that $x<z<y$ holds.
    Denote $\eta$ the unique, up to isomorphism, dense countable linear
    order without end-points, i.e.\ the order of the rationals.
\item A linear order is \emph{scattered} when $\eta$ is not among its
    suborders.
\item A linear order is \emph{$\sigma$-scattered} if it is a countable
    union of scattered suborders.
\item Denote by $\lin$ (resp.\ $\Lin$) the class of all countable (resp.\
    all) linear orders, and by $\scat$ the class of \emph{scattered
    countable} linear orders.
\item An order is \emph{complete} when all its non-empty upper bounded
    subsets have least upper bound.
\item An order is \emph{bounded} if it has both a maximum and a minimum.
\end{itemize}
\end{definitions}

\begin{fact}
Every linear order $K$ can be completed using Dedekind cuts (see
\cite[Definition 2.22]{Rosens1982}).
\end{fact}

\begin{nota}
If $x,y$ are elements of an order $L$ with $x<y$, square bracket notation
will be used to denote the interval they determine. For example $[x,y]=\{
z\mid x\leq z\leq y\} $ denotes the closed interval between $x$ and $y$.
Other similar notations such as $\left]x,y\right[$ or $\left[ x,
{\rightarrow} \right[$ are self-explaining.

The closed interval notation $[x,y]$ will be sometimes used regardless of how
$x$ and $y$ are ordered, meaning in any case the set of all elements between
them.
\end{nota}

The following is a useful description of order preserving functions that are
continuous with respect to the order topologies. Recall that a basis for the
order topology consists of the open intervals, including those of the form
$\left]{\leftarrow },x\right[$ and $\left]x,{\rightarrow }\right[$.

\begin{lemma} \label{continuity}
Let $K,L\in \Lin $ and let $f:K\to L$ be order preserving. Then $f$ is
continuous if and only if for every non-empty subset $A$ of $K$, if the
supremum (or the infimum) of $A$ exists, then the same holds for $f(A)$ and
$f(\sup A)=\sup f(A)$ (or $f(\inf A)=\inf f(A)$).
\end{lemma}

\begin{proof}
Let $f$ be continuous and suppose that $\sup A$ exists, in order to show that
$f(\sup A)=\sup f(A)$ holds. Since $f$ is order preserving, $f(\sup A)$ is an
upper bound for $f(A)$. If there were some upper bound $b$ of $f(A)$ with
$b<f(\sup A)$, then $f^{-1}(\left] b, {\rightarrow} \right[)$ would contain
$\sup A$ but no elements of $A$. Since any neighbourhood of $\sup A$ contains
elements of $A$, this contradicts the continuity of $f$. Similarly for the
greatest lower bound.

Conversely, assume the condition on extrema. Take $b,b'\in L$, with $b<b'$,
to show that $f^{-1}(\left] b,b' \right[)$ is open. Let $a\in K$ be such that
$b<f(a)<b'$ holds. It is enough to prove that if $a$ is not the least element
of $K$ and does not have an immediate predecessor, then there is $c<a$ such
that $b<f(c)$ (and similarly if $a$ is not the last element of $K$ and does
not have an immediate successor). But then letting $A = \left] {\leftarrow},
a \right[ $ one has $a=\sup A$, so that $f(a)=\sup f(A)$ holds, which implies
the claim. A similar argument shows that the preimages of other kinds of
basic open sets are open.
\end{proof}

We recall for convenience the definition of the backwards, the sum and the
product of linear orders.

\begin{definitions}~
\begin{itemize}
\item Given $K\in\Lin$ we call \emph{backwards} or \emph{reversal} of $K$
    and we denote $K^\star$ the order that has the same domain as $K$ and
    such that $x\leq_{K^\star}y$ holds if and only if $y \leq_K x$ does.
\item Let $K$ be in $\Lin$ and for every $i\in K$ take $L_i\in \Lin$. The
    sum $\sum_{i\in K}L_i$ is the set $\{(i,l)\mid i\in K\,\&\,l\in L_i\}$
    ordered lexicographically. As a particular case we shall write finite
    sums as $L_1+\cdots+L_n$.

Notice that, since according to our definition a linear order is non-empty,
whenever we consider a sum it is tacitly assumed that both its index set
and all of its summands are non-empty.
\item Given $K,L\in\Lin$ the product $K\cdot L$ or simply $KL$ is the set
    $K\times L$ ordered antilexicographically.
\end{itemize}
\end{definitions}

\subsection{On better-quasi-orders}

We recall here the definitions of well-quasi-order and of better-quasi-order.

\begin{definitions}\label{def:qobasics}~
\begin{itemize}
\item A \emph{quasi-order}, or qo, is a transitive reflexive relation on
    some set $Q$. We typically write $\leq_Q$ for a qo on $Q$.
\item An infinite sequence $(q_n)$ of elements of $Q$ is {\it bad} if for
    all $n,m$ in $\omega$ such that $n<m$ we have $q_n \nleq_Q q_m$.
\item A qo $(Q,\leq_Q)$ is \emph{well-quasi-ordered}, or wqo, if there are
    no bad sequences.
\item We let $[\omega]^\omega$ denote the set of infinite subsets of
    $\omega$ with the topology induced by the topology on $2^\omega$ under
    the identification of a set with its characteristic function. For
    $X\in[\omega]^\omega$, we let $[X]^\omega$ be the set of infinite
    subsets of $X$.
\item If $Q$ is a set, a $Q$-{\it array} is a function $f$ with domain
    $[X_0]^\omega$ for some $X_0 \in [\omega]^\omega$ and values in $Q$
    such that $f^{-1}(\{y\})$ is open for all $y\in Q$.
\item If $(Q,\leq_Q)$ is a qo, a $Q$-array $f$ is \emph{bad} if for all $X
    \in \dom(f)$ we have $f(X)\nleq_Q f(X^+)$, where $X^+=X\backslash\{\min
    X\}$.
\item A qo $(Q,\leq_Q)$ is a \emph{better-quasi-order} (bqo) if there are
    no bad $Q$-arrays.
\end{itemize}
\end{definitions}

The original definition of bqo is due to Nash-Williams (\cite{nashwell}). The
equivalent definition we gave is due to Simpson (\cite{simpsonbqo}). For more
concerning bqos, see \cite{vEMSbqo, LStRbqo, fbqo}.

\begin{facts}~
\begin{enumerate}
\item Every bad sequence $(q_n)_{n\in\omega}$ in $Q$ induces a bad
    $Q$-array  defined by $f(X)=q_{\min X}$, so that every bqo is indeed a
    wqo.
\item A straightforward application of the Galvin-Prikry theorem shows that
    every finite union of bqos (and in particular any finite qo) is a bqo.
\item The Galvin-Prikry theorem implies also that every finite product of
    bqos is a bqo.
\end{enumerate}
\end{facts}

\begin{definitions}\label{defn:basics}~
Given $K$ and $L$ in $\Lin$, we write
\begin{itemize}
\item $K \inj L$ if there is an order preserving injection from $K$ into
    $L$.
\item $K\injcont L$ if there is an order preserving continuous injection
    from $K$ into $L$.
\end{itemize}
\end{definitions}

It is obvious that $K \injcont L$ implies $K \inj L$.

\begin{definitions}[\cite{LStRbqo}]\label{defn:preservebqos}
Let $\Cal{C}$ be a class of structures and morphisms between them such that
the identities are $\Cal{C}$-morphisms and $\Cal{C}$-morphisms are closed
under composition.
\begin{itemize}
\item Given a qo $Q$, set
\[Q^{\Cal{C}}=\{f\mid f\mbox{ is a function, }\dom(f)\mbox{ is a }\Cal{C}\mbox{-structure, } \im(f)\sub Q\},\]
quasi-ordered as follows
\begin{multline*}
f_0\leq f_1 \Leftrightarrow\exists\,\Cal{C}\mbox{-morphism } g:\dom(f_0)\rao\dom(f_1)\\
\st\forall\,x\in\dom(f_0)\,f_0(x)\leq_Qf_1(g(x)).\end{multline*}
\item $\Cal{C}$ \emph{preserves bqos} if for all bqo $Q$ the class
    $Q^{\Cal{C}}$ is still bqo.

\end{itemize}
\end{definitions}

\begin{facts}~
\begin{enumerate}
\item If a class $\Cal{C}$ of structures preserves bqos then $\Cal{C}$ is a
    bqo under the qo induced by $\Cal{C}$-morphisms.
\item If $\Cal{C}$ preserves bqo then so does any subclass of $\Cal{C}$.
\end{enumerate}
\end{facts}

Using this terminology, Laver's theorem (\cite{laverfraisse}) can be stated
as follows.

\begin{theorem}\label{Laver}
The class of $\sigma$-scattered linear orders under $\inj$ preserves bqos. In
particular $(\lin,{\inj})$ preserves bqos.
\end{theorem}

This result was strengthened by van Engelen, Miller and Steel
(\cite[Theorem 3.5]{vEMSbqo}).

\begin{theorem}\label{vEMS}
The class $(\lin,{\injcont})$ preserves bqos.
\end{theorem}


\subsection{Epimorphisms: definition and first properties} \label{epimorphisms}

Our main object of interest is introduced in the next definition.

\begin{definition}
Let $K,L \in \Lin$. We write $K \strong L$ if there is an order preserving
surjection, also called an \emph{epimorphism}, from $L$ onto $K$. Thus a
witness to the fact that $K \strong L$ holds is a surjective function $g:L\to
K$ such that for all $a,b \in L$ we have $(a \leq_L b \Rightarrow g(a) \leq_K
g(b))$.

Denote by $\equiv$ the induced equivalence relation and by $[K]$ the
equivalence class of $K$ under $\equiv$. We still use $ \strong $ for the
partial order induced on equivalence classes.
\end{definition}

\begin{definition}
If $L \in \Lin$ has no last element, the \emph{cofinality} $cof(L)$ of $L$ is
the least ordinal which is the length of a sequence unbounded above in $L$.
Given $K, L\in\Lin$ a map $f:K \rao L$ is \emph{cofinal} when its range is
unbounded above in $L$.

Similarly, when $L$ has no least element, we define the {\it coinitiality}
$coi(L)$ and \emph{coinitial} maps.
\end{definition}

Notice that $cof(L), coi(L) \leq |L|$ and hence if $L \in \lin$ has no last
element then $cof(L) = \omega$, and similarly for $coi(L)$.

\begin{facts}\label{facts:basics on epi}
Given $K, L \in \Lin$, we have
\begin{enumerate}
\item $K \strong L$ if and only if $L=\sum_{i\in K}L_i$.
\item If $K$ is finite and $|L| \geq |K|$ then $K\strong L$.
\item If $L$ has least (or last) element while $K$ does not, then $K
    \nstrong L$.
\item Let $g$ witness that $K \strong L$ holds then:
\begin{enumerate}
\item $g$ has a right inverse, that is an order preserving embedding of
    $K$ into $L$ and therefore $K \strong L$ implies $K \inj L$;
\item if $K$ does not have least (or last) element, any such right
    inverse is coinitial (or cofinal) in $L$;
\item if $L$ is a complete order, then so is $K$;
\item if $K$ is dense, then $g$ is continuous with respect to the order topology.
\end{enumerate}
\item If $K,L$ are without maximum and $K \strong L$, then $cof(K)=cof(L)$.
    Similarly for the coinitiality of linear orders without minimum.
\end{enumerate}
\end{facts}

\begin{proof}
For (4d) use Lemma \ref{continuity}.

(5) The existing epimorphism $g:L\to K$ grants $cof(K)\leq cof(L)$. The right
inverse of $g$ witnesses $cof(L)\leq cof(K)$.
\end{proof}

\begin{lemma} \label{completevsall}
Let $K$ be a complete order and $L$ be any order. There is an order
preserving surjection $g:L\to K$ if and only if one of the following cases
holds:
\begin{enumerate}
\item $K$ has minimum $a$, a maximum $b$ and there is an order preserving
    injection $f:K\to L$;
\item $K$ has minimum $a$, no maximum and there is an order preserving
    cofinal injection $f:K\to L$;
\item $K$ has maximum $b$, no minimum and there exists an order preserving
    coinitial injection $f:K\to L$;
\item $K$ does not have minimum nor maximum and there exists an order
    preserving, coinitial and cofinal injection $f:K\to L$.
\end{enumerate}
\end{lemma}

\begin{proof}
The necessity of the condition, in each of the four cases, is witnessed by any right inverse $f$ of $g$.
Conversely, for each of the four cases, an epimorphism $g:L\to K$ is built as follows:

(1), (2)
\begin{align*}
g(y)=
\begin{cases}
a & \mbox{if } y<f(a) \\
\sup\{ x\in K\mid f(x)\leq y\} & \mbox{if } y\geq f(a)
\end{cases}
\end{align*}

(3)
\begin{align*}
g(y)=
\begin{cases}
\inf\{ x\in K\mid f(x)\geq y\} & \mbox{if } y\leq f(b) \\
b & \mbox{if } y>f(b)
\end{cases}
\end{align*}

(4)
Fix $c\in K$ and define
\begin{align*}
g(y)=
\begin{cases}
\inf\{ x\in K\mid f(x)\geq y\} & \mbox{if } y<f(c) \\
\sup\{ x\in K\mid f(x)\leq y\} & \mbox{if } y\geq f(c)
\end{cases}
\end{align*}
\end{proof}

Cases 1 and 2 apply in particular when $K$ is a well-order.

\section{The structure of $ \strong $} \label{countableorders}
\subsection{Basic facts}
We start by proving the following three useful propositions.

\begin{prop}\label{propnonscat}
For any $L \in \lin$:
\begin{enumerate}
\item $L \strong \eta$;
\item $L \strong  1+\eta $ if and only if $L$ has a minimum;
\item $L \strong \eta +1$ if and only if $L$ has a maximum;
\item $L \strong 1+\eta +1$ if and only if $L$ has minimum and maximum.
\end{enumerate}
Moreover:
\begin{enumerate}\setcounter{enumi}4
\item $\inf ([1+\eta ],[\eta +1])=[1+\eta +1]$;
\item $\sup ([1+\eta ],[\eta +1])=[\eta ]$ (even when the $\sup$ is taken
    in $\Lin$).
\end{enumerate}
\end{prop}

\begin{proof}
First notice that $\eta$ is isomorphic to $\eta L$ for any $L\in \lin$, so
(1)--(4) follow easily from Fact \ref{facts:basics on epi}.1.

For (5) suppose $L \strong 1+\eta$ and $L \strong\eta +1$, so that $L$ has
both a first and a last element. The assertion then follows from (4).

It remains to prove (6) in $\Lin$. If $1+\eta \strong L$ and $\eta +1 \strong
L$, then $L$ can be written both as a ($1+\eta $)-sum $\Sigma $ and as an
($\eta +1$)-sum $\Sigma'$. Then the first summand in $\Sigma $ can be written
either as an $\eta $-sum or as an ($\eta +1$)-sum. The claim follows as $\eta
+\eta =\eta +1+\eta =\eta $.
\end{proof}

\begin{prop}\label{structnonscat}
Let $L$ be a countable, non-scattered, linear order. Then exactly one of the
following four possibilities holds:
\begin{enumerate}
\item $\eta \strong L$;
\item $L=L_0+ \hat L $, for some unique $L_0$ and $\hat L$ with $L_0$
    scattered and $\eta \strong \hat L$ (in which case $1+\eta \strong L$);
\item $L= \hat L +L_1$, for some unique $L_1$ and $\hat L$ with $L_1$
    scattered and $\eta \strong \hat L$ (in which case $\eta+1 \strong L$);
\item $L= L_0+ \hat L + L_1$, for some unique $L_0$, $L_1$ and $\hat L$
    with $L_0$ and $L_1$ scattered and $\eta \strong \hat L$.
\end{enumerate}
In particular, $1+\eta+1 \strong L$ for any countable, non-scattered $L$.
\end{prop}
\begin{proof}
The four cases are mutually exclusive because $\eta \nstrong K$ for every
scattered $K$.

By \cite[Theorem 4.9]{Rosens1982} $L$ is a sum of scattered orders on a dense
index set which, since $L$ is non-scattered, is one of $\eta$, $1+\eta$,
$\eta+1$ and $1+\eta+1$. Each one of the four cases corresponds to one of the
cases in the statement of the proposition. It remains to prove uniqueness in
the last three cases.

Take case (2) and suppose there are $L_0,L'_0$ scattered and $\hat{L},
\hat{L}'$ above $\eta$ such that $L=L_0+\hat{L}=L'_0+\hat{L}'$ holds. Suppose
$L_0\neq L_0'$, then as both orders are tails of $L$ one is a suborder of the
other, so for instance $L_0\subset L_0'$. Hence $L_0'=L_0+L_0''$ for some
scattered $L_0''$, so $\hat{L}=L_0''+\hat{L}'$ should hold, which is
impossible since we supposed that $\eta\strong\hat{L}$ holds. The other cases
are similar.
\end{proof}

In cases (2)--(4) of Proposition \ref{structnonscat} the suborder $L_0$
(respectively $L_1$) is called the \emph{scattered initial tail}
(respectively the \emph{scattered final tail}) of $L$.

\begin{prop}
The relation $ \strong $ has four minimal elements on infinite orders having
countable coinitiality or a minimum, and countable cofinality or a maximum:
$[\omega ],[\omega +1],[1+\omega^\star],[\omega^\star]$.
\end{prop}

\begin{proof}
The four linear orders $\omega$, $\omega +1$, $1+\omega^\star$,
$\omega^\star$ are pairwise $ \strong $-incomparable, and as they are complete we can
use Lemma \ref{completevsall}. Given $L$, if $L$ does not have a least (or a
last) element, then there is a coinitial decreasing (or a cofinal increasing)
sequence in $L$ and $\omega^\star \strong L$ (or $\omega \strong L$) by
Lemma \ref{completevsall}. Otherwise $L=\{ a\} +L'+\{ b\} $ and $L'$ contains
either a decreasing sequence (in which case $1+\omega^\star \strong L$) or an
increasing sequence (and then $\omega +1 \strong L$), again by Lemma
\ref{completevsall}.
\end{proof}

Lemma \ref{completevsall} allows the following description of the cones above
the aforementioned elements, for generic orders $L$:
\begin{itemize}
\item $\omega \strong L$ if and only if $cof(L) = \omega$;
\item $\omega +1 \strong L$ if and only if there is a bounded countable
    increasing sequence in $L$;
\item $1+\omega^\star \strong L$ if and only if there is a bounded
    countable decreasing sequence in $L$;
\item $\omega^\star \strong L$ if and only if $coi(L) = \omega$.
\end{itemize}

\subsection{The bqo $ \strong $ on $\lin$}
By Lemma \ref{completevsall}.1, in the very special case of complete linear
orders with first and last element any order preserving injection can be
reversed into an order preserving surjection. As a consequence, $ \strong $
is indeed bqo on the fragment of $\lin$ consisting of complete orders with
minimum and maximum.

We are now going to extend this to all countable linear orders using the
completion of any linear order $K$, coloring the elements of the completion
according to whether they already are in $K$ or they represent a gap of $K$,
and making sure that the final order is bounded.

\begin{definition}
Given $L \in \Lin$, define the \emph{closure} of $L$, denoted $\overline{L}$,
as the order obtained by completing $L$ and then possibly adding a first or a
last element in case $L$ does not have them. Let the \emph{complete coloring
of $L$} be the map $c_L: \overline{L} \to 3$ defined by
\[
c_L(x) = \begin{cases}
2& \mbox{if } x\in L;\\
1& \mbox{if } x\in\{\min \overline{L} ,\max \overline{L} \} \mbox{ and } x\notin L; \\
0& \mbox{ otherwise.}
\end{cases}
\]
Let us denote by $\colorcplt$ the order on $3^{(\Lin,\injcont)}$ of
Definition \ref{defn:preservebqos}, where $3$ is quasi-ordered by the
identity.
\end{definition}

Notice that if $L \in \lin$ is non-scattered then $\overline{L} \notin \lin$,
as it contains a copy of $\mathbb{R}$.

The next lemma shows that if the colorings on the closures of two orders are
comparable with respect to $\colorcplt$, the injection can be reversed into
an order preserving surjection between the original orders. This generalizes
the fact we mentioned at the beginning of this section.

\begin{lemma}\label{prop:reversing cplte colorings}
Given $K$ and $L$ in $\Lin$, if $c_K\colorcplt c_L$ then $K\strong L$.
\end{lemma}
\begin{proof}
Fix $K$ and $L$ in $\Lin$, and suppose there exists a continuous, order
preserving injective map $f: \overline{K} \rao \overline{L}$ such that for
all $x\in\overline{K}$ we have $c_K(x)=c_L(f(x))$. In particular, $x\in K$ if
and only if $f(x)\in L$.

The map $f$ admits a canonical dual map $g:\overline{L} \lgrao \overline{K}$
defined by
\[
g(y) =\sup\{x\in\overline{K}\mid f(x) \leq y \}
\]
(this includes the case $g (y)=\min \overline{K} $ whenever $\{x \in
\overline{K} \mid f (x)\leq y\} =\emptyset $). As $g(f(x)) = x$ for every $x
\in \overline{K}$, the map $g$ is a surjective order preserving map from
$\overline{L}$ onto $\overline{K}$. It is now sufficient to prove that
$\im(g\rest{L})=K$ holds.

If $x \in K$ then $f(x) \in L$ and we have $g(f(x))=x$, so that $K \sub
\im(g\rest{L})$ holds.

Let $y \in L$ and suppose towards a contradiction that $g(y) \notin K$. There
are three possible cases:
\begin{itemize}
\item[(a)] there are non-empty sets $A,B\subseteq K$ such that $g(y)=\sup
    A=\inf B$;
\item[(b)] $g(y)=\min \overline{K}$;
\item[(c)] $g (y)=\max \overline{K}$.
\end{itemize}

(a) Notice that $f(a) \leq y$ for every $a \in A$ and hence
\begin{align*}
f(g(y)) &= f(\sup A)\\
& = \sup f(A) \leq y
\end{align*}
by Lemma \ref{continuity}.

On the other hand $f(b) > y$ for every $b \in B$ and hence, using again Lemma
\ref{continuity},
\begin{align*}
f(g(y))&=f(\inf B)\\
&=\inf f(B) \geq y.
\end{align*}

Thus $f(g(y))= y$ holds, against $c_K(g(y)) =0$ and $c_L(y) =2$.

(b) In this case we have $f(x)>y$ for every $x \in \overline{K} \setminus
\{\min \overline{K}\}$. Since $\min \overline{K} = \inf (\overline{K}
\setminus \{\min \overline{K}\})$ Lemma \ref{continuity} implies that $f(\min
\overline{K}) \geq y$. But then, since $c_K (\min \overline{K}) =1$, we must
have $f(\min \overline{K}) = \max \overline{L}$ which is impossible as
$\overline{K}$ has more than one element.

(c) In this case we have $f(x) \leq y$ for every $x \in \overline{K}
\setminus \{\max \overline{K}\}$ and, arguing as in (b), we obtain $f(\max
\overline{K}) = \min \overline{L}$, which is also a contradiction.
\end{proof}

We can now prove our main result.

\begin{theorem}\label{mainbqos}
The qos $(\scat,{\strong})$ and $(\lin,{\strong})$ are bqos.
\end{theorem}

\begin{proof}
Recall that if $L \in \lin$ then $L$ is scattered if and only if $L$ has
countably many initial intervals (\cite[\S6.7]{Fra00}). Hence if $L \in
\scat$ then $\overline{L}$ is countable and complete, so that $\overline{L}
\in \scat$. By Lemma \ref{prop:reversing cplte colorings} the map $\Phi :
\scat \to 3^{ \scat },K\mapsto c_K$ satisfies $\Phi (K) \colorcplt \Phi
(L)\Rightarrow K \strong L$. But using Theorem \ref{vEMS},
$(3^{(\scat,\injcont)}, \colorcplt)$ is bqo, and finally so is
$(\scat,{\strong})$.

Now it will be shown that each of the classes of linear orders corresponding
to the four cases of Proposition \ref{structnonscat} is a bqo under $ \strong
$. The linear orders falling in case (1) constitute a unique $\equiv$-class,
so they form a bqo. For the orders in case (2), assign to each such $L$ its
scattered initial tail $L_0$. So, for $L,M$ in this class, one has $L_0
\strong M_0\Rightarrow L \strong M$; since we already proved that $( \scat ,{
\strong })$ is a bqo, this shows that this class is a bqo. Similarly for case
(3). Finally, to each $L$ satisfying case (4), assign the pair $(L_0,L_1)$ of
its scattered initial and final tails. So, for $L,M$ in this class, $L_0
\strong M_0\wedge L_1 \strong M_1\Rightarrow L \strong M$; since $( \scat ,{
\strong })$ is a bqo and bqos are closed under finite products, this
establishes that $ \strong $ is a bqo for the orders in case (4) too.

Since bqos are closed under finite unions, this allows to conclude that $( \lin ,{ \strong })$ is a bqo.
\end{proof}

\subsection{Preserving bqos}
Next, one could ask if $ \strong $ preserves bqos. Notice that to be
meaningful, Definition \ref{defn:preservebqos} cannot be taken verbatim,
since otherwise any $\strong$-strictly increasing sequence would provide a
decreasing sequence in $Q^{\lin}$, for any qo $Q$. In any reasonable
adaptation of the definition, the roles of $f_0$ and $f_1$ should be switched
and the existence of a surjection $g: \dom(f_1) \rao \dom (f_0)$ should be
required. The first definition that comes to mind is probably the following.

\begin{definition}
Given a qo $Q$ the class $Q^{( \Lin ,\strong)}$ is quasi-ordered by setting
$f_0 \strong' f_1$ if and only if there exists an order preserving surjection
$g: \dom (f_1) \rao \dom (f_0)$ such that $\forall y \in \dom(f_1)\,
f_0(g(y)) \leq_Q f_1(y)$.
\end{definition}

However, even finite orders do not preserve bqos for this notion.

\begin{prop}
The qo $2^{(\omega,\strong')}$ (where the elements $0$ and $1$ of $2$ are
incomparable) admits an infinite antichain.
\end{prop}
\begin{proof}
For $n>0$ let $s_n$ be the sequence that alternates 0's and 1's of length
$2n$. Take $m,n$ two integers with $0<m<n$, then $s_n \nstrong's_m$ since
$m<n$. Fix any order preserving surjection $g:n\rao m$, as $m<n$ there is an
integer $i<n$ such that $g(i)=g(i+1)$, but $s_n(i)\neq s_n(i+1)$ so $g$
cannot witness that $s_m\strong's_n$. Consequently $(s_n)_{n\in\omega}$ is an
infinite antichain.
\end{proof}

To find a better definition observe that $f_0 \strong' f_1$ if and only if
for every $x \in \dom (f_0)$
\[
\forall y \in \dom (f_1) (g(y)=x \implies f_0(x) \leq_Q f_1(y)).
\]
Now notice that the displayed formula is equivalent to $\{f_0(x)\}
\leq_Q^\sharp f_1(g^{-1}(x))$ where $f_1(g^{-1}(x)) = \{f_1(y) \mid y \in
\dom(f_1) \land g(y)=x\}$ and $\leq_Q^\sharp$ is sometimes called the Smyth
quasi-order: for $A,B \in\Pow(Q)$, $A \leq_Q^\sharp B$ if and only if
$\forall b \in B\, \exists a \in A\, a \leq_Q b$.

There is another natural quasi-order on $\Pow(Q)$, which is variously known
as the domination quasi-order, the Egli-Milner quasi-order, or the Hoare
quasi-order: for $A,B \in\Pow(Q)$, $A \leq_Q^\flat B$ if and only if $\forall
a \in A\,\exists b \in B\, a \leq_Q b$ (both $\leq_Q^\sharp$ and
$\leq_Q^\flat$ have been studied from the viewpoint of wqo and bqo theory in
\cite{powqo}). If we ask that $\{f_0(x)\} \leq_Q^\flat f_1(g^{-1}(x))$ for
every $x \in \dom (f_0)$ we obtain the following definition, which we will
show makes $\lin$ with surjections preserve bqos.

\begin{definition}
Given a qo $Q$ the class $Q^{( \Lin ,\strong)}$ is quasi-ordered by setting
$f_0 \strong^Q f_1$ if and only if there exists an order preserving
surjection $g: \dom (f_1) \rao \dom (f_0)$ such that
\[
\forall x \in \dom (f_0)\, \exists y \in \dom (f_1) (g(y)=x \land f_0(x) \leq_Q f_1(y)).
\]
\end{definition}

When $f\in Q^{( \Lin ,\strong)}$ has domain $L$, it will be often convenient to stress this fact by denoting $f=(L,f)$.

\begin{definition}
Let $Q$ be a qo and let $\overline{Q}$ be the disjoint union of $Q$ with two
mutually incomparable elements $0$ and $1$. For any $f=(L,f) \in
Q^{(\Lin,\strong)}$ we define the \emph{closure} of $f$, denoted
$ \overline f =(\overline{L}, \overline f )$, as the element of
$\overline{Q}^{(\Lin,\strong)}$ defined as follows. The order $\overline{L}$
is obtained by completing $L$ and then possibly adding a first or a last
element in case $L$ does not have them. The coloring $ \overline f $ of
$\overline{L}$ is defined by
\[
\overline f (x) = \begin{cases}
f(x)& \mbox{if } x\in L;\\
1& \mbox{if } x\in\{\min \overline{L} ,\max \overline{L} \} \mbox{ and } x\notin L; \\
0& \mbox{ otherwise.}
\end{cases}
\]
Let us denote by $\color$ the order on $\overline{Q}^{(\Lin,{\injcont})}$ of
Definition \ref{defn:preservebqos}.
\end{definition}

\begin{lemma}\label{prop:reversing cplte coloringsQ}
Given $(K,f_0)$ and $(L,f_1)$ in $Q^{(\Lin,\strong)}$, if $(\overline{K},
\overline{f_0} ) \color (\overline{L}, \overline{f_1})$ then $(K,f_0)
\strong^Q (L,f_1)$.
\end{lemma}
\begin{proof}
The proof is essentially the same as the proof of Lemma \ref{prop:reversing
cplte colorings}. Given $f$ witnessing $(\overline{K}, \overline{f_0}) \color
(\overline{L}, \overline{f_1})$, we define $g$ and prove that $g\rest{L}$ is
an order preserving surjection onto $K$ exactly as before. Since $g(f(x)) =
x$ and $f_0(x) \leq_Q f_1(f(x))$ for every $x \in K$ we have $(K,f_0)
\strong^Q (L,f_1)$.
\end{proof}

The theorem we obtain from Lemma \ref{prop:reversing cplte coloringsQ} could
be used to obtain the first part of Theorem \ref{mainbqos} as a corollary.

\begin{theorem}\label{presbqos}
The class $(\scat,{\strong})$ preserves bqos, i.e.\ if $Q$ is a qo then so is
$Q^{(\scat,\strong)}$ under $\strong^Q$.
\end{theorem}
\begin{proof}
Exactly as the first part of the proof of Theorem \ref{mainbqos}, using Lemma
\ref{prop:reversing cplte coloringsQ} in place of Lemma \ref{prop:reversing
cplte colorings}.
\end{proof}

Notice however that the second part of the proof of Theorem \ref{mainbqos}
(dealing with $\lin$ in place of $\scat$) does not go through in this case.
We do not know whether the result can be extended to $ \lin $.

\section{Description of $ \strong $ on some special classes of orders} \label{descriptionofspecialclasses}

\subsection{$ \strong $ on ordinals}
When restricted to ordinal numbers, the structure of relation $ \strong $
admits a neat description.

\begin{prop} \label{strongonordinals}~
\begin{enumerate}
\item Let $\alpha =\omega^{\gamma_0}n_0+\ldots +\omega^{\gamma_k} n_k$,
    $\beta =\omega^{\delta_0}m_0+\ldots +\omega^{\delta_h}m_h$ be limit
    ordinals (that is $\gamma_k,\delta_h>0$) in Cantor normal form. Then
$$\alpha \strong \beta \sse \alpha\leq\beta\wedge cof(\alpha )=cof(\beta ) \wedge \gamma_k\leq\delta_h .$$
\item If $\alpha $ is a successor ordinal and $\beta $ is any ordinal,
then
$\alpha \strong \beta \sse \alpha\leq\beta $.
\item If $\alpha $ is a limit ordinal and $\beta $ is a successor ordinal,
    then $\alpha \nstrong \beta $.
\end{enumerate}
\end{prop}

\begin{proof}~
\begin{enumerate}
\item Assume that $\alpha \strong \beta $ holds. Then so does
    $\alpha\leq\beta $ as there exists an increasing, cofinal injection
    $f:\alpha\to\beta $. Also we have $cof(\alpha )=cof(\beta )$ by Fact
    \ref{facts:basics on epi}.5. Moreover $f$ maps the last occurrence of
    $\omega^{\gamma_k}$ in the Cantor normal form of $\alpha$ cofinally
    into $\beta $. This implies that a final interval of this
    $\omega^{\gamma_k}$ is mapped increasingly into $\omega^{\delta_h}$. By
    indecomposability $\omega^{\gamma_k}$ embeds into $\omega^{\delta_h}$,
    so $\gamma_k \leq \delta_h$.

Conversely, assume that $\alpha\leq\beta$, $cof(\alpha )=cof(\beta )$ and
$\gamma_k\leq\delta_h$ hold. Then we have $\alpha
=\alpha'+\omega^{\gamma_k}$ and $\beta =\alpha'+\beta'+\omega^{\delta_h}$
for some ordinals $\alpha',\beta'$. To apply Lemma \ref{completevsall}.2 it
suffices to show that there is an increasing cofinal
$f:\omega^{\gamma_k}\to\omega^{\delta_h}$. Let $\rho =cof(\alpha
)=cof(\beta )=cof(\omega^{\gamma_k})=cof(\omega^{\delta_h})$ and let
$\varphi :\rho\to\omega^{\gamma_k}$, $\psi :\rho\to\omega^{\delta_h}$ be
cofinal and increasing, such that $\varphi (0)=0$ holds and $\varphi $ is
continuous at limits. Now define inductively a new increasing and cofinal
function $\psi':\rho\to\omega^{\delta_h}$ by letting
\begin{itemize}
\item $\psi'(0)=0$,
\item $\psi'$ continuous at limits, and
\item $\psi'(\xi+1)=\max (\psi (\xi +1),\psi'(\xi)+(\varphi (\xi +1)-\varphi (\xi )))$,
\end{itemize}
where for $\sigma, \tau$ ordinals with $\tau \leq \sigma$, their difference
$\sigma-\tau$ is defined as the unique ordinal $\lambda $ such that $\tau +
\lambda = \sigma$. This is possible, since $\varphi(\xi+1) - \varphi(\xi)<
\omega^{\gamma_k} \leq \omega^{\delta_h}$. For each $\zeta<
\omega^{\gamma_k}$ there exist uniquely determined $\xi<\rho$ and $\tau<
\varphi(\xi+1) - \varphi(\xi)$ such that $\zeta = \varphi(\xi) + \tau$
holds. Set $f(\zeta) = \psi'(\xi)+ \tau$.

\item By Lemma \ref{completevsall}.1.

\item By Fact \ref{facts:basics on epi}.3.\qedhere
\end{enumerate}
\end{proof}

\begin{cor} \label{iimpliesstrong}
Let $\beta$ be an ordinal. Then $\alpha \strong \beta $ for every non-null
$\alpha\leq\beta $ if and only if $\beta$ is countable and a finite multiple
of an indecomposable ordinal.
\end{cor}

\begin{proof}
Let $\beta $ be countable and finite multiple of an indecomposable ordinal,
that is $\beta =\omega^{\delta }m$ for some $m>0$. Then every non-null
$\alpha\leq\beta $ is either a successor ordinal or it has countable
cofinality and it has Cantor normal form $\alpha =\omega^{\gamma_0}n_0+\ldots
+\omega^{\gamma_h}n_h$, with $\gamma_h\leq\delta $. Now apply Proposition
\ref{strongonordinals}.

On the other hand, if $\beta $ is uncountable there are limit ordinals less
than $\beta $ with different cofinalities, so there cannot be an epimorphism
of $\beta $ onto each of them. Finally, if $\beta$ is not finite multiple of
an indecomposable ordinal, then it has Cantor normal form $\beta
=\omega^{\delta_0}m_0+\ldots +\omega^{\delta_h}m_h$ with $h\geq 1$ and, by
Proposition \ref{strongonordinals} there cannot be an epimorphism from $\beta
$ onto $\omega^{\delta_0}m_0$.
\end{proof}

\subsection{Exploiting completeness}
Some of the ideas used in previous sections can be employed to find an
explicit description of $\strong$ on some other classes of linear orders.

\begin{definition}
According to \cite{Rosens1982}, if $L \in \lin$ and $x\in L$ let $c(x)=\{
y\in L\mid [x,y] \mbox{ is finite} \} $ be the \emph{condensation} of $x$.
Let also $L^1=\{ c(x)\}_{x\in L}$ with the natural order. This is the
\emph{condensation} of $L$.
\end{definition}

We first consider the class of complete bounded $\sigma$-scattered linear
orders. Given such an $L$, define a coloured linear order $(L',\varphi_L)$ on
the set of colours $\{1, 2, 3, \ldots, {\leftarrow}, {\rightarrow}\}$. These
colours are ordered by $\sqsubseteq$ which is the usual order relation on
$\{1, 2, 3, \ldots\}$, is such that $n \sqsubseteq {\leftarrow}$ and $n
\sqsubseteq {\rightarrow}$, while $\leftarrow$ and $\rightarrow$ are
incomparable.

The order $L'$ is obtained from $L^1$ by replacing each condensation class of
order type $\zeta$ with two consecutive elements, the members of a pair of
intervals of order types $\omega^\star$ and $\omega$, respectively, of which
the class is the union.

We now need to define $\varphi_L$. Given $x\in L'$, there are various
possibilities:
\begin{itemize}
\item if $x \in L^1$ is a finite condensation class, then $\varphi_L(x) =
    |x|$;
\item if $x \in L^1$ is a condensation class of order type $\omega^\star$
    or $\omega $, then $\varphi_L(x)$ is $\leftarrow$ or $\rightarrow$,
respectively;
\item If $x$ is one of the two intervals replacing a condensation class
    $x'\in L^1$ of order type $\zeta$, let $\varphi_L(x)$ be $\leftarrow$
    or $\rightarrow$ according to whether it is the first or the second of
    the two elements in the order $L'$.
\end{itemize}

\begin{prop}
Given complete bounded $\sigma$-scattered linear orders $K$ and $L$, if there
is an order preserving injection $g:L'\to K'$ with $\forall x\in L'\
\varphi_L(x)\sqsubseteq\varphi_K (g(x))$ then $L \strong K$.
\end{prop}
\begin{proof}
Under the given hypotheses, we are going to define an epimorphism $h:K\to L$.

Given $y\in K'$, if $y$ is in the range of $g$, say $g(x)=y$, define $h$ on
$y$ as any order preserving surjection onto $x$. Otherwise, there are three
cases. If $y$ is less than every element in the range of $g$, define $h$ on
$y$ as the constant function with value the least element of $L$. Similarly,
if $y$ majorizes the range of $g$, let $h$ on $y$ be constant of value the
maximum of $L$. Finally, suppose $g$ takes values both smaller and bigger
than $y$. Then $h$ on $y$ will be constant with value $\sup \bigcup \{ x\in
L'\mid g(x)<y\} $.
\end{proof}

By Theorem \ref{Laver} we obtain the following.

\begin{cor} \label{completebounded}
When restricted to complete, bounded, $\sigma $-scattered orders, the
relation $ \strong $ is bqo.
\end{cor}

\begin{theorem}
When restricted to complete scattered linear orders with fixed coinitiality
and cofinality, the relation $ \strong $ is bqo.
\end{theorem}

\begin{proof}
Fix regular cardinals $\alpha ,\beta$.
It will be shown that each of the following classes forms a bqo.
\begin{itemize}
\item[(1)] Complete scattered linear orders with minimum and cofinality
    $\beta $.
\item[(2)] Complete scattered linear orders with maximum and coinitiality
    $\alpha$.
\item[(3)] Complete scattered linear orders with coinitiality $\alpha$ and
    cofinality $\beta $.
\end{itemize}
First remark that if $L$ is a scattered ordering, then given any two points
$x_1<x_2$ in $L$, there are  consecutive $y_1,y_2\in L$ with $x_1\leq
y_1<y_2\leq x_2$.

(1) Let $L$ be complete and scattered, with minimum and cofinality $\beta$.
There exists an increasing cofinal sequence $\{ \ell_{\xi }\}_{\xi <\beta }$
in $L$ with $\ell_0=\min L$ and such that if $\xi $ a successor ordinal the
element $\ell_{\xi }$ has an immediate predecessor in $L$, while if $\xi $ is
a limit then $\ell_{\xi }=\sup\{ \ell_{\rho }\}_{\rho <\xi }$. Indeed, fix
any cofinal increasing sequence $\{ \ell_{\xi }'\}_{\xi <\beta }$; by
possibly modifying it, it can be assumed $\ell_0'=\min L$ and that, for $\xi
$ a limit, $\ell_{\xi }'=\sup\{ \ell_{\rho }'\}_{\rho <\xi }$. Let
$\ell_0=\ell_0'$. Suppose $\{ \ell_{\xi }\}_{\xi <\gamma }$ has been defined
with the required properties and in such a way that $\forall\xi <\gamma\
\ell_{\xi }'\leq \ell_{\xi }$, in order to define $\ell_{\gamma }$. Remark
that $\{ \ell_{\xi }\}_{\xi <\gamma }$ is bounded in $L$. If $\gamma $ is
limit, let $\ell_{\gamma }=\sup\{ \ell_{\xi }\}_{\xi <\gamma }$; in
particular, $\ell_{\gamma }'\leq \ell_{\gamma }$. For $\gamma =\mu +1$
successor, let $\delta <\beta $ be least such that $\ell_{\delta }'>
\ell_{\mu }$. Find consecutive points $y_1,y_2\in L$ with $\ell_{\delta
}'\leq y_1<y_2\leq \ell_{\delta +1}'$ and let $\ell_{\gamma }=y_2$.

So $L$ is a $\beta $-sum of complete orders $L_{\xi}$ with least and last
element: $L_\xi$ has end points $\ell_{\xi }$ and the immediate predecessor
of $\ell_{\xi +1}$. Let $\varphi_L$ be the colouring of $\beta $ which maps
each $\xi <\beta $ to $L_{\xi }$.

Let $L$ and $M$ be complete and scattered, with minimum and cofinality
$\beta$. We claim that if there is an embedding $f$ of $\beta $ into itself
such that for all $\xi <\beta $ we have $L_{\xi } \strong M_{f(\xi )}$ then
$L \strong M$. Indeed, fix epimorphisms $g_{\xi }:M_{f(\xi )}\to L_{\xi }$,
and define $g: M \to L$ as follows. If $\gamma = f(\xi)$ for some $\xi<\beta$
let $g \rest {M_\gamma} = g_{\xi }$, while if $\gamma <\beta $ is not in the
range of $f$ there is a least $\delta <\beta $ such that $\gamma <f(\delta
)$: let $g \rest {M_{\gamma }}$ be constant with value $\min L_{\delta }$.

Using Corollary \ref{completebounded} and Theorem \ref{Laver} we obtain the
conclusion.

(2) and (3) are proved similarly.
\end{proof}

\bibliographystyle{alpha}
\bibliography{epimorphism}

\end{document}